\documentclass{amsart}
\usepackage{amssymb}

\newcommand{\N}{{\mathbb{N}}}
\newcommand{\Z}{{\mathbb{Z}}}
\newcommand{\C}{{\mathbb{C}}}
\newcommand{\R}{{\mathbb{R}}}
\newcommand{\T}{{\mathbb{T}}}

\let\Re=\undefined\DeclareMathOperator*{\Re}{Re}

\DeclareMathOperator*{\lcm}{lcm}

\newtheorem{theorem}{Theorem}[section]

\newtheorem{lemma}[theorem]{Lemma}

\newtheorem{proposition}[theorem]{Proposition}
\theoremstyle{definition}
\newtheorem{definition}[theorem]{Definition}

\newtheorem{remark}[theorem]{Remark}

\newcounter{smalllist}

\numberwithin{equation}{section}
\newcommand{\qtq}[1]{\quad\text{#1}\quad}
\allowdisplaybreaks

\newcommand{\eps}{{\varepsilon}}

\begin{document}

\title[Scale invariant Strichartz estimates on tori and applications]{Scale invariant Strichartz estimates on tori\\and applications}

\author{Rowan Killip and Monica Vi\c{s}an}

\address
{Rowan Killip\\
Department of Mathematics\\
University of California, Los Angeles, CA 90095, USA}
\email{killip@math.ucla.edu}

\address
{Monica Vi\c{s}an\\
Department of Mathematics\\
University of California, Los Angeles, CA 90095, USA}
\email{visan@math.ucla.edu}

\begin{abstract}
We prove scale-invariant Strichartz inequalities for the Schr\"odinger equation on rectangular tori (rational or irrational) in all dimensions.
We use these estimates to give  a unified and simpler treatment of local well-posedness of the energy-critical nonlinear Schr\"odinger equation
in dimensions three and four. 
\end{abstract}

\maketitle

%\setcounter{tocdepth}{2}
%\tableofcontents

%%%%%%%%%%%%%%%%%%%%%%%%%%%%%%%%%%%%%%%%%%%%%%%%%%%%%%%%%%%%%%%%%%%%%%%%%%%%%%%%%%%%%%%%%%%%%%%%%%%%%%%%%%%%%%%%%%%%%%%%%%%%%%
\section{Introduction}
%%%%%%%%%%%%%%%%%%%%%%%%%%%%%%%%%%%%%%%%%%%%%%%%%%%%%%%%%%%%%%%%%%%%%%%%%%%%%%%%%%%%%%%%%%%%%%%%%%%%%%%%%%%%%%%%%%%%%%%%%%%%%%

The most general flat torus is formed as the quotient of $\R^d$ by a lattice.  In this paper, we will only consider rectangular tori,
namely, those of the form
$
\R^d / ( L_1\Z \times L_2\Z \times\cdots\times L_d \Z^d)
$
with $L_1,\ldots,L_d\in(0,\infty)$. 
Our goal is to prove certain space-time estimates for solutions of the linear Schr\"odinger equation on such manifolds.

Notationally, it will be simpler to fix the base space to be $\T^d:=\R^d/\Z^d$ and to incorporate the geometry of the torus into
the definition of the Laplacian.  Put differently, we use coordinates based on the standard torus and then use the Laplace--Beltrami operator
associated to the induced metric, that is,
$$
\Delta:= \sum_{j=1}^d \theta_j  \frac{\partial^2\ }{\partial x_j^2}, \qtq{or equivalently,} \widehat{\Delta f} (k):= - \sum_{j=1}^d \theta_j^{ } k_j^2 \hat f(k).
$$
Here $\theta_j = L_j^{-2}$ and we employ the following convention for the Fourier transform:
$$
\hat f (k) = \int_{\T^d} e^{-2\pi i k x} f(x) \,dx
\qtq{so that}
f (x) = \sum_{k\in \Z^d} e^{2\pi i k x} \hat f(k).
$$

With these notations, the solution $u(t,x)$ to the linear Schr\"odinger equation with initial data $u_0(x)$ is given by
\begin{align}\label{Fseries}
u(t,x) = e^{it\Delta} u_0 = \sum_{k\in \Z^d} \exp\Bigl\{2\pi i \bigl[k x-t{\textstyle\sum_{j=1}^d} \theta_j^{ } k_j^2\bigr]\Bigr\} \widehat{u_0}(k).
\end{align}
Note that by making a change of variables in time, there is no loss of generality to assume that $\theta_1, \ldots, \theta_d\in(0,1]$.

The main result of this paper is the following:

\begin{theorem}[Scale-invariant Strichartz estimates]\label{T:Strichartz}
Fix $d\geq 1$, $\theta_1, \ldots, \theta_d\in (0,1]$, $1\leq N\in 2^{\Z}$, and $p>\frac{2(d+2)}{d}$.  Then
\begin{equation}\label{E:T:Strichartz}
\|e^{it\Delta}P_{\leq N} f\|_{L_{t,x}^p([0,1]\times\T^d)}\lesssim N^{\frac d2- \frac{d+2}p}\|f\|_{L_x^2},
\end{equation}
where $\Delta:=\theta_1 \partial_{x_1}^2 + \cdots + \theta_d \partial_{x_d}^2$.
\end{theorem}

Unlike $\R^d$, the torus $\T^d$ does not admit a true scaling symmetry; however, for very short times, the linear evolution of highly concentrated initial data will not distinguish the two.
For well-posedness questions of nonlinear problems, such concentrated solutions are the principal adversary.   Correspondingly, scale-invariant estimates are an essential tool for treating nonlinear problems at the critical regularity.

We should also note that by choosing $p$ close to $\frac{2(d+2)}{d}$ one may make the $\frac d2- \frac{d+2}p$
loss of derivatives as small as one wishes.  It is not difficult to verify that the stated estimate fails for the square torus (i.e., $\theta_1=\cdots=\theta_d=1$) if one takes $p=\frac{2(d+2)}{d}$;
see \cite{Bourgain:TorusStrichartz}.

Very recently, Bourgain and Demeter (see \cite[Theorem~2.4]{BourgainDemeter:l2}) proved analogous Strichartz estimates with an arbitrarily small loss of scaling:

\begin{theorem}[Non-scale-invariant Strichartz estimates]\label{T:epsStrichartz}
Fix $d\geq 1$, $\theta_1, \ldots, \theta_d\in (0,1]$, $1\leq N\in 2^{\Z}$, and $p\geq\frac{2(d+2)}{d}$.  Then for any $\eta>0$,
\begin{equation}\label{E:T:epsStrichartz}
\|e^{it\Delta}P_{\leq N} f\|_{L_{t,x}^p([0,1]\times\T^d)}\lesssim_\eta N^{\frac d2- \frac{d+2}p+\eta}\|f\|_{L_x^2},
\end{equation}
where $\Delta:=\theta_1 \partial_{x_1}^2 + \cdots + \theta_d \partial_{x_d}^2$.
\end{theorem}

This result will be an essential part of the proof of Theorem~\ref{T:Strichartz}.  In much earlier work, Bourgain showed that in the case of a square torus,
Theorem~\ref{T:epsStrichartz} implies Theorem~\ref{T:Strichartz}; see \cite[Proposition~3.113]{Bourgain:TorusStrichartz}.

The space-time Fourier methods used by Bourgain for the square torus are ill-suited to the case of an irrational torus.  We will be using the basic dispersive estimate
for the propagator (see Lemma~\ref{L:K_N estimates}), which pushes all the difficulty into bounding the resulting temporal convolution.  This style of argument (which is
closer to the usual Euclidean treatment) is indifferent to the rational/irrational character of the $\theta$s. In particular, the intricate astigmatism resulting from refocusing
at slightly different times in each coordinate direction can be brutishly handled by the arithmetic-geometric mean inequality.  Nonetheless, important insights employed by Bourgain in \cite{Bourgain:TorusStrichartz} do inform and suffuse our treatment of the subtle temporal convolution.

We now give a brief summary of prior work on Strichartz estimates on square and irrational tori:\\
$\bullet$\;In \cite{Bourgain:TorusStrichartz}, Bourgain considered only the square torus.  He proved Theorems~\ref{T:Strichartz} and~\ref{T:epsStrichartz} in dimensions one and two.  He also proved \eqref{E:T:Strichartz} for $p>4$ when $d=3$ and for $p\geq\frac{2(d+4)}{d}$ when $d\geq 4$.\\
$\bullet$\;The paper \cite{Bourgain:TorusIrrat} of Bourgain was the first to consider irrational tori. It considers only the case $d=3$ and proves scale-invariant
$L^p_t L^4_x$ Strichartz estimates for $p>\frac{16}3$.\\
$\bullet$\;Bourgain, \cite{Bourgain:Curved}, and Demeter, \cite{Demeter:Incidence}, gave very different proofs that \eqref{E:T:epsStrichartz} holds for $p=\frac{2(d+3)}{d}$ on all tori.\\
$\bullet$\;The paper \cite{GuoOhWang} of Guo, Oh, and Wang proves several Strichartz estimates on irrational tori.  In particular, they obtain \eqref{E:T:Strichartz}
in the following cases: $d=2$ and $p>\frac{20}{3}$, $d=3$ and $p>\frac{16}{3}$, $d=4$ and $p>4$, and lastly, $d\geq 5$ and $p=4$.  The also prove that \eqref{E:T:Strichartz}
holds for $d=3$ and $p>\frac{14}{3}$ under the additional assumption $\theta_1=\theta_2$.

As an application of Theorem~\ref{T:Strichartz} we consider the initial-value problem for the energy-critical nonlinear Schr\"odinger equation
\begin{equation}\label{NLS}
\begin{cases}
i\partial_t u + \Delta u= \pm |u|^{\frac4{d-2}}u\\
u(0)=u_0\in H^1(\T^d)
\end{cases}
\end{equation}
in spatial dimensions $d\in\{3,4\}$.  Specifically, we show the following:

\begin{theorem}[Well-posedness for the energyy-critical NLS]\label{T:EC}
Fix $d\in\{ 3,4\}$ and let $u_0\in H^1(\T^d)$.  Then there exists a time $T=T(u_0)$ and a unique solution $u\in C_t ([0, T);H^1(\T^d)) \cap X^1([0, T))$ to \eqref{NLS}.  Moreover, there exists $\eta_0=\eta_0(d)>0$ such that if $\|u_0\|_{H^1(\T^d)}\leq \eta$, then the solution $u$ is global in time.
\end{theorem}

In the three dimensional case, this theorem is not new.  The result was proved in \cite{HTT} for the case of the square torus, in \cite{GuoOhWang}
for the case when $\theta_1=\theta_2$, and for the fully irrational torus in \cite{Strunk}.  Here, we will combine the new estimates provided by Theorem~\ref{T:Strichartz}
with several beautiful ideas introduced in \cite{HTT} to provide a significantly simpler proof.  We use bilinear estimates rather than trilinear estimates; moreover, we do not need to exploit the temporal orthogonality of free evolutions to prove the bilinear estimate we use.

In four dimensions, Theorem~\ref{T:EC} was proved in \cite{HTT:4d}, but only in the case of a square torus.  Again, the new estimates provided by Theorem~\ref{T:Strichartz}
allow for a simpler argument.  In particular, we do not need any subtle multilinear estimates.

\subsection*{Acknowledgements}  R.~K. was supported by NSF grant DMS-1265868. M.~V. was supported by the Sloan Foundation and NSF grant DMS-1161396.  We are both
indebted to the Hausdorff Institute of Mathematics, which hosted us during our work on this project.  We are also grateful to Sebastian Herr for helpful conversations.

%%%%%%%%%%%%%%%%%%%%%%%%%%%%%%%%%%%%%%%%%%%%%%%%%%%%%%%%%%%%%%%%%%%%%%%%%%%%%%%%%%%%%%%%%%%%%%%%%%%%%%%%%%

\subsection{Notation and useful lemmas}
Throughout this text, we will be regularly referring to the spacetime norms
\begin{equation}\label{E:qr def}
\bigl\| u \bigr\|_{L^p_tL_x^r([0,1]\times\T^d)}
    := \biggl(\int_{[0,1]} \biggl(\int_{\T^d} |u(t,x)|^r\,dx\biggr)^{p/r} \,dt\biggr)^{1/p},
\end{equation}
with obvious changes if $p$ or $r$ are infinity.

We write $X\lesssim Y$ to indicate that $X\leq C Y$ for some constant $C$, which is permitted to depend on the
ambient spatial dimension, $d$, without further comment.

Let $\phi$ be a smooth radial cutoff on $\R$ such that $\phi(x)=1$ for $|x|\leq 1$ and $\phi(x)=0$ for $|x|\geq 2$.  With $N\in 2^\N$ we define the Littlewood--Paley projections
\begin{equation}\label{E:LPd}
\begin{gathered}
\widehat{P_1 f}(k) := \widehat{f_1}(k):= \hat{f}(k)\prod_{j=1}^d\phi(k_j),
	\quad \widehat{P_{\leq N} f}(k):=\widehat{f_{\leq N}}(k):=\hat{f}(k)\prod_{j=1}^d\phi\bigl(\tfrac {k_j}N\bigr),\\
\text{and} \quad \widehat{P_Nf}(k) :=\widehat{f_N}(k):= \hat f(k)\prod_{j=1}^d\bigl[\phi\bigl(\tfrac {k_j}N\bigr)- \phi\bigl(\tfrac{2k_j}N \bigr)\bigr],
\end{gathered}
\end{equation}
where $k=(k_1, \ldots, k_d)\in \Z^d$.  Using Littlewood--Paley projectors with this product structure simplifies the proof of Theorem~\ref{T:Strichartz} slightly.

Next we recall the definition of the function spaces $U^p$ and $V^p$ and use them to construct the relevant function spaces for our applications.  The general theory of $U^p$ and $V^p$ spaces is discussed at some length in \cite{KochTatVis}; we will confine ourselves here to reviewing the definitions and basic properties in the specific setting that is relevant to our problem.  In particular, we only consider finite time intervals of the form $[0,T)$.  Let $H$ be a separable Hilbert space over $\C$; in this paper, this will be
$\C$ or $H^s(\T^d)$ with $s=0,1$.  Let $\mathcal Z$ be the set of finite partitions $0=t_0 <t_1 <\ldots<t_K \leq T$.  We use the convention that $v(T) := 0$ for all functions $v :[0,T)\to H$.

\begin{definition}
Let $1\leq p<\infty$.  An $U^p$-atom is a function $a:[0,T)\to H$ of the form
$$
a=\sum_{k=1}^K \chi_{[t_{k-1},t_k)}\phi_{k-1},
$$
where $\{t_k\}\in \mathcal Z$ and $\{\phi_k\}\subset H$ with $\sum_{k=0}^{K-1} \|\phi_k\|_H^p=1$.  The atomic space $U^p([0,T); H)$ is the space of all functions $u:[0,T)\to H$ of the form
$$
u=\sum_{j=1}^\infty \lambda_ja_j
$$
with $\{\lambda_j\}\in \ell^1(\C)$ and $a_j$ being $U^p$-atoms.  The norm on $U^p([0,T); H)$ is given by
$$
\|u\|_{U^p}:=\inf\Bigl\{ \sum_{j=1}^\infty |\lambda_j|:\ u=\sum_{j=1}^\infty \lambda_ja_j \text{ with } \{\lambda_j\}\in \ell^1(\C) \text{ and $U^p$-atoms } a_j\Bigr\}.
$$
\end{definition}

\begin{definition}
Let $1\leq p<\infty$.  The space $V^p([0,T);H)$ is the space of all functions $v:[0,T)\to H$ such that
$$
\|v\|_{V^p}:=\sup_{\{t_k\}\in \mathcal Z}\Bigl( \sum_{k=1}^K \|v(t_k)-v(t_{k-1})\|_H^p\Bigr)^{1/p}<\infty.
$$
The space $V^p_{rc}([0,T); H)$ denotes the closed subspace of all right-continuous functions $v:[0,T)\to H$ such that $v(0) = 0$.
\end{definition}

\begin{remark}
The spaces $U^p([0,T); H)$, $V^p([0,T); H)$, and $V^p_{rc}([0,T); H)$ are Banach spaces and satisfy
$$
U^p([0,T); H)\hookrightarrow V^p_{rc}([0,T); H)\hookrightarrow U^q([0,T); H)\hookrightarrow L^\infty([0,T); H)
$$
for all $1\leq p<q<\infty$.
\end{remark}

\begin{definition}
Let $s=0,1$.  Then $U_{\Delta}^p H^s$ and $V_\Delta^pH^s$ denote the spaces of all functions $u : [0,T)\to H^s(\T^d)$ such that the map $t\to e^{-it\Delta}u(t)$ is in $U^p([0,T);H^s)$ and $V^p([0,T); H^s)$, repectively, with norms given by
$$
\|u\|_{U_{\Delta}^p H^s}:= \|e^{-it\Delta}u\|_{U^p([0,T);H^s)} \qtq{and}  \|u\|_{V_{\Delta}^p H^s}:= \|e^{-it\Delta}u\|_{V^p([0,T);H^s)}.
$$
We define $X^s([0,T))$ and $Y^s([0,T))$ to be the spaces of all functions $u:[0,T)\to H^s(\T^d)$ such that for every $\xi\in \Z^d$ the map $t\to \widehat{e^{-it\Delta} u(t)}(\xi)$ is in $U^2([0,T); \C)$ and $V^2_{rc}([0,T); \C)$, respectively, with norms given by
\begin{align*}
\|u\|_{X^s([0,T))}:=\Bigl( \sum_{\xi\in \Z^d}\langle\xi\rangle^{2s}\| \widehat{e^{-it\Delta} u(t)}(\xi)\|^2_{U^2 }\Bigr)^{1/2}, \\
\|u\|_{Y^s([0,T))}:=\Bigl(\sum_{\xi\in \Z^d}\langle\xi\rangle^{2s}\| \widehat{e^{-it\Delta} u(t)}(\xi)\|^2_{V^2}\Bigr)^{1/2}.
\end{align*}
These are the same spaces used in \cite{HTT} and subsequent works.
\end{definition}

\begin{remark}\label{R}
We have the continuous embeddings $U_{\Delta}^2 H^s\hookrightarrow X^s\hookrightarrow Y^s\hookrightarrow V_{\Delta}^2 H^s$.  We also note that
$$
\|u\|_{L_t^\infty H^s_x([0,T)\times\T^d)}\lesssim \|u\|_{X^s([0,T))}
$$
and
$$
\Bigl\|\int_0^t e^{i(t-s)\Delta}F(s)\, ds\Bigr\|_{X^s([0,T))}\lesssim \|F\|_{L_t^1 H^s_x([0,T)\times\T^d)}.
$$
\end{remark}

Using the atomic structure of $U^p$ and Remark~\ref{R}, we can recast the Strichartz estimates from Theorem~\ref{T:Strichartz} as follows:
\begin{align}\label{UStrichartz}
\|P_{\leq N} u\|_{L^p([0,T)\times\T^d)}\lesssim N^{\frac d2-\frac{d+2}p}\|P_{\leq N} u\|_{U^p_\Delta L^2} \lesssim N^{\frac d2-\frac{d+2}p}\|P_{\leq N} u\|_{Y^0([0,T))}
\end{align}
for all $p>\tfrac{2(d+2)}d$ and $N\geq 1$.  In particular, due to the Galilei invariance of solutions to the linear Schr\"odinger equation,
\begin{align}\label{UCStrichartz}
\|P_Cu\|_{L^p([0,T)\times\T^d)}\lesssim N^{\frac d2-\frac{d+2}p}\|P_Cu\|_{Y^0([0,T))} \qtq{for all} p>\tfrac{2(d+2)}d
\end{align}
and for any cube $C\subset \R^d$ of side-length $N\geq 1$.

%%%%%%%%%%%%%%%%%%%%%%%%%%%%%%%%%%%%%%%%%%%%%%%%%%%%%%%%%%%%%%%%%%%%%%%%%%%%%%%%%%%%%%%%%%%%%%%%%%%%%%%%%%%%%%%%%%%%%%%%%%%%%%
\section{Scale invariant Strichartz estimates}
%%%%%%%%%%%%%%%%%%%%%%%%%%%%%%%%%%%%%%%%%%%%%%%%%%%%%%%%%%%%%%%%%%%%%%%%%%%%%%%%%%%%%%%%%%%%%%%%%%%%%%%%%%%%%%%%%%%%%%%%%%%%%%

The implicit constants in this section will be allowed to depend on the magnitude of $\{\theta_j\}_{j=1}^d$ and we will not be tracking that dependence.  It is worth noting however that the number theoretical properties of $\{\theta_j\}_{j=1}^d$ play no role in our arguments.

In this section, we will write
\begin{equation}\label{K_N defn}
K_N(t,x):=[e^{it\Delta} P_{\leq N} \delta_0](x) = \sum_{k\in \Z^d} \prod_{j=1}^d \phi(k_j) e^{2\pi i [ x_j k_j -  t \theta_j^{ } k_j^2]}
\end{equation}
for the convolution kernel associated to the frequency-localized propagator.  Here $1\leq N\in 2^\Z$ and $P_{\leq N}$ is the Littewood--Paley projector defined in \eqref{E:LPd}.

As observed already in \cite{Bourgain:TorusStrichartz}, the passage from Theorem~\ref{T:epsStrichartz} to Theorem~\ref{T:Strichartz} requires additional information about the action of
the convolution kernel $K_N(t,x)$ only in places where it is large.  This will become apparent when we complete the proof of Theorem~\ref{T:Strichartz} at the end of this section.

From the micro-local perspective, we expect $K_N$ to be large only near conjugate points of the geodesic flow.  For the square torus, there are many such conjugate points, one at every rational time;
however, the degree of refocusing is governed by the denominator of the rational number concerned.  These heuristics are borne out by Lemma~\ref{L:K_N estimates} below, whose statement is best understood in
the context of Dirichlet's Lemma on rational approximation.

\begin{lemma}[Dirichlet]\label{L:Dirichlet}
Given an integer $N\geq 2$ and $\beta\in[0,1]$, there exist integers $1\leq q<N$ and $0\leq a\leq q$ so that $(a,q)=1$ and
$
| \beta - \tfrac aq | \leq \tfrac{1}{Nq}.
$
\end{lemma}

Recall that $(a,q)$ denotes the greatest common divisor of $a$ and $q$; correspondingly, $(a,q)=1$ asserts that $a$ and $q$ are relatively prime.  Note also that $(0,q)=q$.

\begin{lemma}[Dispersive estimate for $K_N$]\label{L:K_N estimates}
Choosing integers $0\leq a_j\leq q_j< N$ so that $(a_j,q_j)=1$ and $|\theta_jt-\frac {a_j}{q_j}|\leq\frac1{q_jN}$, we have
\begin{align*}
|K_N(t,x)|\lesssim\prod_{j=1}^d \frac{N}{\sqrt{q_j} \bigl(1+ N\bigl|\theta_jt-\frac {a_j}{q_j}\bigr|^{1/2}\bigr)}
\end{align*}
uniformly for $t\in[0,1]$.
\end{lemma}

Due to the product structure of \eqref{K_N defn}, the $d$-dimensional estimate is an immediate corollary of the one-dimensional case.  This in turn follows from an application of Weyl's method
(estimating the square modulus via a linear change of variables).  See \cite[Lemma~3.18]{Bourgain:TorusStrichartz} for further details.

Incidentally, Lemma~\ref{L:K_N estimates} shows that $K_N$ can only be very large if $\{\theta_j t\}_{j=1}^d$ can all be simultaneously well-approximated by rationals with small denominator.
Correspondingly, under a mild Diophantine condition, which holds for Lebesgue almost all $d$-tuples of parameters $\theta_j$, one may show that $K_N$ is very large only very close to $t=0$.
This leads to a much shorter proof of Theorem~\ref{T:Strichartz} for such $d$-tuples.

So far, we have been rather nebulous about what it means for $K_N(t,x)$ to be very large.  It turns out that the precise meaning depends on the exponent $p$ from Theorem~\ref{T:Strichartz} that one
is treating.  For now, we will use a parameter $0<\sigma\ll1$ that will be chosen later and say that $K_N(t,x)$ is large when $t$ belongs to
$$
\mathcal{T}:= \bigl\{ t\in[0,1]:\, q_j N^2 \bigl|\theta_jt-\tfrac {a_j}{q_j}\bigr| \leq N^{2\sigma} \text{ for some $j$, $q_j\leq N^{2\sigma}$, and $(a_j,q_j)=1$}\bigr\}.
$$
We then define
$$
\tilde K_N(t,x) := \chi_{\mathcal{T}} (t)  K_N(t,x).
$$
In view of Lemma~\ref{L:K_N estimates}, this construction guarantees that
\begin{align}\label{E:diff bound}
|K_N(t,x) - \tilde K_N(t,x)| \lesssim N^{d(1-\sigma)}.
\end{align}

The centerpiece of our analysis is the following proposition, which establishes space-time estimates for $\tilde K_N$.

\begin{proposition}[Strichartz estimates for $\tilde K_N$]\label{P:Strichartz}
Choose $2<p,r \leq \infty$ such that $\frac d2 - \frac2p-\frac dr>0$.  Then
\begin{align*}
\|\tilde K_N*F\|_{L_t^p L_x^r ([0,1]\times\T^d)}\lesssim N^{2(\frac d2-\frac 2p-\frac dr)} \|F\|_{L_t^{p'} L_x^{r'} ([0,1]\times\T^d)},
\end{align*}
provided $\sigma$ is sufficiently small {\upshape(\!}depending on $(d,p,r)$ only{\upshape)}.
\end{proposition}

As we will see, Proposition~\ref{P:Strichartz} is a direct consequence of the next two lemmas.  The first lemma concerns mapping properties of $\tilde K_N$ as a convolution kernel on $\T^d$ (with $t$ fixed);
this will follow easily from Lemma~\ref{L:K_N estimates}.  The second lemma is much more challenging and deals with the resulting temporal convolution.  This two-step argument has strong parallels to the standard
approach in the Euclidean setting, where one uses the (much simpler) dispersive estimate and then the time convolution is handled very swiftly by an application of the Hardy--Littlewood--Sobolev inequality.
Such an approach yields only very poor estimates in the torus setting.  It is essential to exploit the non-resonant structure of the temporal convolution kernel which yields substantial gains for large $q$
relative to the Hardy--Littlewood--Sobolev inequality.

To state the first lemma, we introduce a family of smooth radial cutoffs on $\R$ as follows:
\begin{equation*}
\phi_{N^{-2}} (x):=
\begin{cases}
1, \quad \text{if } |x|\leq 1\\
0, \quad \text{if } |x|\geq 2
\end{cases}
\end{equation*}
and for all dyadic $T> N^{-2}$ we define $\phi_T(x):= \phi_{N^{-2}} (x)- \phi_{N^{-2}} (2x)$.  Exploiting just these definitions, we have
\begin{align}\label{E:covers T}
\sum_{j=1}^d \ \ \sum_{Q=1}^{N^{2\sigma}} \ \sum_{T=N^{-2}}^{N^{2\sigma-2}/Q} \sum_{\substack{(a,q)=1\\q\sim Q}} \phi_T\bigl( \tfrac{\theta_jt-\frac aq}T\bigr) \geq  1 \qtq{for all} t\in\mathcal T.
\end{align}
Here, and in all that follows, $Q$ and $T$ are restricted to lie in $2^\Z$ and $q\sim Q$ means that $Q\leq q < 2Q$.

\begin{lemma}[Dispersive estimates for $\tilde K_N$]\label{L:dispersive}
For $t\in[0,1]$ and $2\leq r\leq \infty$ we have
\begin{align*}
\|\tilde K_N(t)*f\|_{L^r(\T^d)}\lesssim  \|f\|_{L^{r'}(\T^d)}\sum_{j=1}^d \ \sum_{Q=1}^{N^{2\sigma}} \sum_{T=N^{-2}}^{N^{2\sigma-2}/Q} (QT)^{\frac dr-\frac d2}
		\sum_{\substack{(a,q)=1\\q\sim Q}} \phi_T\bigl( \tfrac{\theta_jt-\frac aq}T\bigr).
\end{align*}
\end{lemma}

\begin{proof}
By the unitarity of the propagator $e^{it\Delta}$, we have
\begin{align*}
\|K_N(t)*f\|_{L^2(\T^d)}= \|f\|_{L^2(\T^d)}.
\end{align*}
On the other hand, from the kernel estimates of Lemma~\ref{L:K_N estimates}, we obtain
\begin{align*}
\|K_N(t)*f\|_{L^\infty(\T^d)}\lesssim \|f\|_{L^1(\T^d)} \prod_{j=1}^d \frac{N}{\sqrt{q_j} \bigl(1+ N\bigl|\theta_jt-\frac {a_j}{q_j}\bigr|^{1/2}\bigr)},
\end{align*}
where $0\leq a_j\leq q_j< N$ obey $(a_j,q_j)=1$ and $|\theta_jt-\frac {a_j}{q_j}|\leq\frac1{q_jN}$.

Interpolating between these two bounds and using the arithmetic--geometric mean inequality, we derive that for any $2\leq r\leq \infty$,
\begin{align*}
\|K_N(t)*f\|_{L^r(\T^d)}
&\lesssim \|f\|_{L^{r'}(\T^d)}\prod_{j=1}^d \Biggl(\frac{N}{\sqrt{q_j} \bigl(1+ N\bigl|\theta_jt-\frac {a_j}{q_j}\bigr|^{1/2}\bigr)}\Biggr)^{1-\frac2r}\\
&\lesssim \|f\|_{L^{r'}(\T^d)} \sum_{j=1}^d \Bigl(N^{-2} q_j \bigl(1+ N^2\bigl|\theta_jt-\tfrac{a_j}{q_j}\bigr| \bigr) \Bigr)^{\frac dr-\frac d2}.
\end{align*}
The lemma now follows easily from \eqref{E:covers T}.
\end{proof}

To continue, for fixed $Q$ we define
\begin{align*}
\mathcal F_{1,Q}(t):= \sum_{\substack{(a,q)=1\\q\sim Q}} \delta \bigl(t-\tfrac aq\bigr) \qtq{and}
	\mathcal F_{2,Q}(t):= \sum_{\substack{0\leq a<q\\q\sim Q}} \delta \bigl(t-\tfrac aq\bigr).
\end{align*}
Note that we may write
\begin{align*}
\sum_{\substack{(a,q)=1\\q\sim Q}} \phi_T\bigl( \tfrac{\theta_j t-\frac aq}T\bigr) = \bigl[\mathcal F_{1,Q} * \phi_T \bigl(\tfrac{\cdot}T\bigr)\bigr](\theta_jt)
\end{align*}
and so, by Lemma~\ref{L:dispersive},
\begin{align*}
&\|\tilde K_N(t)*F\|_{L_t^pL_x^r([0,1]\times\T^d)}\\
&\lesssim  \sum_{j=1}^d \ \sum_{Q=1}^{N^{2\sigma}}\ \sum_{T=N^{-2}}^{N^{2\sigma-2}/Q}\  (QT)^{\frac dr -\frac d2}
		\Bigl\|\bigl[\mathcal F_{1,Q} * \phi_T \bigl(\tfrac{\cdot}T\bigr)\bigr](\theta_jt) *\|F(t)\|_{L^{r'}(\T^d)}\Bigr\|_{L_t^p([0,1])},
\end{align*}
for any $2\leq p,r\leq \infty$.

To prove Proposition~\ref{P:Strichartz}, we need to estimate the time convolution in the expression above.  We are going to do this in two steps.  First, we bound convolution with $\mathcal F_{1,Q}*\phi_T(\cdot/T)$ as
an operator on the torus $\T$; in particular, functions will be understood to be periodic in time.  Later, we will reintroduce $\theta_j$ and pass to the requisite convolution on the subset $[0,1]$ of the real line.
We now turn to the first part of this program.

\begin{lemma}\label{L:time convolution torus}
Fix $2<p\leq \infty$.  Then for any $\sigma<\min\{\frac12,1-\frac2p\}$,
\begin{align*}
\Bigl\|\mathcal F_{1,Q} * \phi_T\bigl( \tfrac {\cdot}T\bigr) * f\Bigr\|_{L^p(\T)}\lesssim Q^{\frac2p(1+\eps)}T^{\frac2p} \|f\|_{L^{p'}(\T)}
	\qtq{with} \eps= \tfrac{\sigma(3-2\sigma)}{(1-\sigma)(1-2\sigma)},
\end{align*}
uniformly for  $1\leq Q\leq N^{2\sigma}$ and $N^{-2}\leq T\leq N^{2\sigma-2}/Q$.
\end{lemma}

As the convolution kernel $\mathcal F_{1,Q} * \phi_T$ is positive, we may bound the norm by replacing $\mathcal F_{1,Q}$ by $\mathcal F_{2,Q}$.  The advantage of doing so is that the Fourier transform
of $\mathcal F_{2,Q}$ is more easily and more efficiently estimated than that of $\mathcal F_{1,Q}$; one should compare what follows with \cite[Lemma~3.33]{Bourgain:TorusStrichartz}.

\begin{lemma}[Fourier transform of $\mathcal F_2$]\label{L:F_2}
Let $d_Q(n)$ denote the number of divisors $q$ of $n$ that obey $q\sim Q$.  Then
\begin{align}
\bigl|\widehat{\mathcal F_{2,Q}}(\omega)\bigr|\lesssim Q d_Q(\omega) \qtq{for all} \omega\neq 0
\end{align}
and clearly,
\begin{align*}
\bigl|\widehat{\mathcal F_{2,Q}}(\omega)\bigr|\lesssim Q^2 \qtq{for all} \omega\in \Z.
\end{align*}
\end{lemma}

\begin{proof}
Recall that $\sum_{a=0}^{q-1} e^{2\pi i a \omega/q} = q$ if $q$ divides $\omega$, but vanishes otherwise.  Thus,
\begin{align*}
\widehat{\mathcal F_{2,Q}}(\omega)= \sum_{q\sim Q} \sum_{a=0}^{q-1} e^{2\pi i a \omega /q} = \sum_{q\sim Q} q \chi_{\{q \,|\, \omega\}}
\end{align*}
and the claims immediately follow.
\end{proof}

The proof of Lemma~\ref{L:time convolution torus}, will also rely on a distributional estimate for $d_Q(n)$.  The bound we need can be found in Lemma~4.28 of \cite{Bourgain:Lambda(p)}; for
completeness, we will recapitulate the proof here (with minor modifications).

\begin{lemma}\label{L:bound on divisors}
For any $\alpha, \tau>0$ we have
\begin{align*}
\#\{1\leq n\leq R:\, d_Q(n)>D\}\lesssim_{\tau, \alpha} D^{-2\alpha} Q^{2\tau} R.
\end{align*}
\end{lemma}

\begin{proof}
It suffices to treat the case where $2\alpha=:k$ is an integer.

Observe first that for fixed $q_1,\ldots,q_k$ we have
\begin{align*}
\#\{1\leq n\leq R:\, q_j | n \text{ for all $1\leq j \leq k$}\} &= \#\{1\leq n\leq R:\, \lcm(q_1,\ldots,q_k) | n \} \\
&\leq R / \lcm(q_1,\ldots,q_k).
\end{align*}
On the other hand, by the trivial sub-polynomial bound (see \cite[Theorem~315]{HardyWright}) on the total number of divisors function $d(\cdot)$, we have
$$
\# \{ (q_1,\ldots,q_k) : \lcm(q_1,\ldots,q_k) = \ell \} \leq d(\ell)^k \lesssim_\eps \ell^{k\eps},
$$
for any $\eps>0$.  Correspondingly, by Chebyshev's inequality,
\begin{align*}
\#\{1\leq n\leq R:\, d_Q(n)>D\}
&\lesssim D^{-k}\sum_{n=1}^R\Bigl(\sum_{q\sim Q}\chi_{q\Z}(n) \Bigr)^k\\
&\lesssim D^{-k} \sum_{q_1, \ldots q_k\sim Q}\frac R{\lcm(q_1,\ldots,q_k)}\\
&\lesssim_\eps D^{-k}  \sum_{\ell=1}^{(2Q)^k}  \tfrac{R}{\ell} \ell^{k\eps} \lesssim_{\eps}  D^{-k}  R Q^{\eps k}
\end{align*}
for any $\eps>0$.  The lemma now follows by choosing $\eps<2\tau/k$.
\end{proof}

We now have all the ingredients we need to complete the proof of Lemma~\ref{L:time convolution torus}.

\begin{proof}[Proof of Lemma~\ref{L:time convolution torus}]
We first note that for distinct pairs $(a_1,q_1)$ and $(a_2,q_2)$ such that $(a_1,q_1)=1=(a_2,q_2)$ and $q_1\sim Q\sim q_2$ we have
$$
\bigl|\tfrac{a_1}{q_1}-\tfrac{a_2}{q_2} \bigr|\gtrsim \tfrac1{Q^2}\gg T,
$$
because $\sigma<\frac12$.  Thus
\begin{align*}
\bigl\| \mathcal F_{1,Q} * \phi_T \bigl( \tfrac {\cdot}T\bigr) \bigr\|_{L^\infty(\T)} \leq 1,
\end{align*}
and so
\begin{align}\label{L1Linfty bound}
\Bigl\|\mathcal F_{1,Q} * \phi_T\bigl( \tfrac {\cdot}T\bigr) * f\Bigr\|_{L^\infty(\T)}\lesssim \|f\|_{L^1(\T)}.
\end{align}

Next we will prove a restricted weak type $(r_0' ,r_0)$ estimate for suitable $r_0\in(2,4)$.  The lemma will follow by interpolating between this bound and \eqref{L1Linfty bound}.

Fix $r_0>2$ and take $E,F\subseteq \T$.   Majorizing $\mathcal F_{1,Q}$ by $\mathcal F_{2,Q}$ and employing the Plancherel identity and Young's convolution inequality, we obtain
\begin{align}\label{restricted 1}
\langle \chi_E, \mathcal F_{1,Q} * \phi_T\bigl( \tfrac {\cdot}T\bigr) * \chi_F\rangle
&\lesssim |E|^{\frac12}|F|^{\frac12}\bigl\|\widehat{\mathcal F_{2,Q}}(\omega)T\widehat{\phi_T}(T\omega)\bigr\|_{\ell_\omega^\infty(\omega \text{ good})}\notag\\
&\quad + |E|^{\frac34}|F|^{\frac34}\bigl\|\widehat{\mathcal F_{2,Q}}(\omega)T\widehat{\phi_T}(T\omega)\bigr\|_{\ell_\omega^2(\omega \text{ bad})},
\end{align}
where we declare
\begin{align*}
\text{$\omega\in\Z$ is \emph{good} if and only if $\bigl|\widehat{\mathcal F_{2,Q}}(\omega)\bigr|\leq Q^{1+\delta}A$}
\end{align*}
for some small $\delta>0$ and some $A>0$ to be chosen later.  By definition,
\begin{align}\label{good freq}
\bigl\|\widehat{\mathcal F_{2,Q}}(\omega)T\widehat{\phi_T}(T\omega)\bigr\|_{\ell_\omega^\infty(\omega \text{ good})}\lesssim Q^{1+\delta} A T.
\end{align}

We now turn to estimating the `bad' frequencies.  By Lemma~\ref{L:F_2}, for a `bad' frequency $\omega\neq 0$ we must have $d_Q(\omega)\gtrsim AQ^{\delta}$.  Therefore, using the fact that $\widehat{\phi_T}$ has rapid decay
uniformly in $T$ and Lemma~\ref{L:bound on divisors}, we obtain
\begin{align}
\bigl\|\widehat{\mathcal F_{2,Q}}(\omega)T\widehat{\phi_T}(T\omega)\bigr\|_{\ell_\omega^2(\omega \text{ bad})}^2
&\lesssim T^2Q^4 + \sum_{2^\Z\ni R \geq T^{-1}} \ \sum_{\substack{0<|\omega|\leq R\\ \omega\text{ bad}}} \bigl|\widehat{\mathcal F_{2,Q}}(\omega)\bigr|^2T^2(RT)^{-100}\notag\\
&\lesssim T^2Q^4 +\sum_{2^\Z\ni R \geq T^{-1}} Q^4 A^{-2\alpha}Q^{2\tau-2\alpha \delta}R T^2(RT)^{-100}\notag\\
&\lesssim T^2Q^4\bigl( 1 + T^{-1}A^{-2\alpha}Q^{2\tau-2\alpha \delta}\bigr). \label{bad freq}
\end{align}

We choose
\begin{align}\label{choices}
A:=\bigl(\tfrac{|E||F|}{T^2}\Bigr)^{\frac12-\frac1{r_0}}, \quad  \alpha:=\tfrac{4-r_0}{2(r_0-2)}, \quad \delta:=\alpha^{-1}, \qtq{and} \tau:=\delta.
\end{align}
Using that $|E|,|F|\leq 1$ and the restrictions on $T$ and $Q$, we find
$$
T^{-1} A^{-2\alpha} \geq T^{-\frac{2(r_0-2)}{r_0}}\geq Q^{\frac{2(r_0-2)}{r_0\sigma}} \geq Q^2
$$
provided $r_0 \geq 2/(1-\sigma)$. Thus, combining \eqref{restricted 1}, \eqref{good freq}, and \eqref{bad freq} yields
\begin{align}\label{restricted 2}
\langle \chi_E, \mathcal F_{1,Q} * \phi_T\bigl( \tfrac {\cdot}T\bigr) * \chi_F\rangle \lesssim (|E||F|)^{\frac1{r_0'}} Q^{1+\delta} T^{\frac2{r_0}} .
\end{align}
This proves that this convolution operator satisfies a restricted weak type $(r_0', r_0)$ estimate with bound $Q^{1+\delta} T^{2/{r_0}}$, provided $\sigma$ and
$r_0$ obey the restriction stated above.

Now for $p$ and $\sigma$ as in the hypotheses of the lemma, $r_0=2/(1-\sigma)$ obeys $2<r_0<p$.  Interpolating between \eqref{L1Linfty bound} and the restricted weak type $(r_0', r_0)$ estimate above, we deduce that
$$
\Bigl\|\mathcal F_{1,Q} * \phi_T\bigl( \tfrac {\cdot}T\bigr) * f\Bigr\|_{L^p(\T)}\lesssim Q^{\frac2p( 1 + \eps)}T^{\frac2p} \|f\|_{L^{p'}(\T)},
$$
which proves the lemma.
\end{proof}

The next result performs the second step in the program laid out above, namely, it allows us to pass from convolution on $\T$ to convolution on $\R$.

\begin{lemma}\label{L:time convolution R}
Let $2<p\leq \infty$ and assume $g:\T\to [0, \infty)$ is the kernel of a bounded convolution operator from $L^{p'}(\T)$ to $L^p(\T)$ with norm $A$.  Then for any $\theta\in(0,1]$ and $R>0$ we have
\begin{align*}
\Bigl\|\int_{-R}^R g(\theta s) f(t-s)\,ds\Bigr\|_{L^p(\R)}\lesssim A\theta^{-\frac2p}(1+\theta R) \|f\|_{L^{p'}(\R)}.
\end{align*}
\end{lemma}

\begin{proof}
We argue by duality.  Pick $k\in \Z$ such that $\theta R\leq k$.  Using the hypothesis and the fact that $p>2$, for $h\in L^{p'}(\R)$ we estimate
\begin{align*}
\Bigl|\int_\R \overline{h(t)}& \int_{-R}^R g(\theta s) f(t-s)\,ds\, dt\Bigr|\\
&\leq \iint_{|t-s|\leq R} |h(t)| |f(s)| g(\theta[t-s])\,ds\, dt\\
&\leq \theta^{-2} \iint_{|t-s|\leq k} \bigl|h\bigl(\tfrac t\theta\bigr)\bigr| \bigl|f\bigl(\tfrac s\theta\bigl)\bigl| g(t-s)\,ds\, dt\\
&\leq \theta^{-2} \sum_{\substack{m,n\in \Z\\ |m-n|\leq k+1}}\int_0^1\int_0^1 \bigl|h\bigl(\tfrac {t+n}\theta\bigr)\bigr| \bigl|f\bigl(\tfrac {s+m}\theta\bigl)\bigl| g(t-s)\,ds\, dt\\
&\leq \theta^{\frac2{p'}-2} A \sum_{\substack{m,n\in \Z\\ |m-n|\leq k+1}} \|h\|_{L^{p'}([\frac n\theta, \frac{n+1}\theta])}\|f\|_{L^{p'}([\frac m\theta, \frac{m+1}\theta])}\\
&\leq \theta^{-\frac2{p}} A (2k+3) \Bigl( \sum_{n\in \Z}\|h\|_{L^{p'}([\frac n\theta, \frac{n+1}\theta])}^{p'}\Bigr)^{1/{p'}}
	\Bigl( \sum_{m\in \Z}\|f\|_{L^{p'}([\frac m\theta, \frac{m+1}\theta])}^{p}\Bigr)^{1/{p}}\\
&\leq \theta^{-\frac2{p}} A (2k+3) \|h\|_{L^{p'}(\R)}\Bigl( \sum_{m\in \Z}\|f\|_{L^{p'}([\frac m\theta, \frac{m+1}\theta])}^{p'}\Bigr)^{1/{p'}}\\
&\leq \theta^{-\frac2{p}} A (2k+3) \|h\|_{L^{p'}(\R)}\|f\|_{L^{p'}(\R)}.
\end{align*}
This completes the proof of the lemma.
\end{proof}

We now have all the necessary ingredients to prove Strichartz estimates for the kernel $\tilde K_N$.

\begin{proof}[Proof of Proposition~\ref{P:Strichartz}]
Fix $p,r$ as in the statement of the proposition.  Combining Lemmas~\ref{L:dispersive}, \ref{L:time convolution torus}, and~\ref{L:time convolution R}, we obtain
\begin{align*}
\|\tilde K_N(t)*F&\|_{L_t^pL_x^r([0,1]\times\T^d)}\\
&\lesssim  \sum_{j=1}^d \  \sum_{1\leq Q\leq N^{2\sigma}} \sum_{T=N^{-2}}^{N^{2\sigma-2}/Q}\!\!\! (QT)^{\frac dr -\frac d2} Q^{\frac2p(1+\eps)}T^{\frac2p}	\|F\|_{L_t^{p'}L_x^{r'}([0,1]\times\T^d)}\\
&\lesssim N^{2(\frac d2-\frac 2p-\frac dr)} \|F\|_{L_t^{p'} L_x^{r'} ([0,1]\times\T^d)},
\end{align*}
provided we take $\sigma>0$ sufficiently small so that $\sigma<\min\{\frac12,1-\frac2p\}$ and $\frac d2 -\frac2p(1+\eps)- \frac dr>0$.
This completes the proof of the proposition.
\end{proof}

\begin{proof}[Proof of Theorem~\ref{T:Strichartz}]
Fix $p>p_0:=\frac{2(d+2)}{d}$ and let $f\in L^2(\T^d)$ be normalized via $\|f\|_{L^2(\T^d)}=1$.  By Bernstein's inequality,
$$
\|e^{it\Delta}P_{\leq N} f\|_{L_{t,x}^\infty([0,1]\times\T^d)}\leq C N^{\frac d2} \|e^{it\Delta}P_{\leq N} f\|_{L_t^\infty L_x^2([0,1]\times\T^d)}\leq C N^{\frac d2}
$$
for some $C>0$.  Thus, we may write
\begin{align}\label{to estimate}
\|e^{it\Delta}P_{\leq N} &f\|_{L_{t,x}^p([0,1]\times\T^d)}\notag\\
&= \int_0^\infty p\lambda^{p-1}\bigl|\{(t,x)\in [0,1]\times\T^d:\, \bigl|(e^{it\Delta}P_{\leq N} f)(x) \bigr|>\lambda\} \bigr|\, d\lambda\notag\\
&= \int_0^{CN^{\frac d2}} p\lambda^{p-1}\bigl|\{(t,x)\in [0,1]\times\T^d:\, \bigl|(e^{it\Delta}P_{\leq N} f)(x) \bigr|>\lambda\} \bigr|\, d\lambda.
\end{align}

For most values of $\lambda$, we exploit the non-scale-invariant Strichartz estimates of Bourgain and Demeter recorded in Theorem~\ref{T:epsStrichartz}.  Specifically, for small $\delta>0$ to be chosen later,
this theorem together with Chebyshev's inequality yields
\begin{align}\label{small lambda}
\int_0^{N^{\frac d2-\delta}} p\lambda^{p-1}&\bigl|\{(t,x)\in [0,1]\times\T^d:\, \bigl|(e^{it\Delta}P_{\leq N} f)(x) \bigr|>\lambda \}\bigr|\, d\lambda\notag\\
&\lesssim \int_0^{N^{\frac d2-\delta}} p\lambda^{p-1} \frac{ N^{p_0\eta}}{\lambda^{p_0}}\, d\lambda
\lesssim N^{p(\frac d2-\frac{d+2}p) + p_0\eta-\delta(p-p_0)}\lesssim N^{p(\frac d2-\frac{d+2}p)},
\end{align}
provided we take $\eta<\delta(p-p_0)/p_0$.  This renders acceptable the contribution of $\lambda\leq N^{\frac d2-\delta}$ to the RHS\eqref{to estimate}.

It remains to estimate the contribution of large values of $\lambda$.  To this end, fix $\lambda>N^{\frac d2-\delta}$ and let
$$
\Omega:= \{(t,x)\in [0,1]\times\T^d:\, \bigl|(e^{it\Delta}P_{\leq N} f)(x) \bigr|>\lambda \}.
$$
By choosing some $\omega\in\{0, \frac \pi2, \pi, \frac{3\pi}2\}$ appropriately, we have that
$$
\Omega_\omega:= \{(t,x)\in [0,1]\times\T^d:\, \Re\bigl(e^{i\omega} e^{it\Delta}P_{\leq N} f\bigr)(x)>\tfrac\lambda2 \}
$$
satisfies $|\Omega|\leq 4|\Omega_\omega|$.  By the definition of $\Omega_\omega$ and Cauchy--Schwarz,
\begin{align}
\lambda^2|\Omega_\omega|^2
&\lesssim \Bigl|\int_0^1\int_{\T^d} (e^{it\Delta}P_{\leq N} f)(x)\chi_{\Omega_\omega}(t,x)\, dx\, dt\Bigr|^2 \notag\\
%&\lesssim \Bigl|\bigl\langle f, \int_0^1 e^{-it\Delta}P_{\leq N}\chi_{\Omega_\omega}(t)\, dt\bigr\rangle\Bigr|^2\\
&\lesssim \|f\|_{L^2(\T^d)}^2 \Bigl\| \int_0^1 e^{-it\Delta}P_{\leq N}\chi_{\Omega_\omega}(t)\, dt \Bigr\|_{L^2(\T^d)}^2\notag\\
&\lesssim \int_{\T^d} \int_0^1\int_0^1 \chi_{\Omega_\omega}(t,x)\overline{\bigl[e^{i(t-s)\Delta}P_{\leq N}^2\chi_{\Omega_\omega}(s)\bigr]}(x)\,ds\,dt\, dx\notag\\
&\lesssim \langle \chi_{\Omega_\omega}, K_N \chi_{\Omega_\omega}\rangle_{L^2_{t,x}}. \label{E:omE}
\end{align}

To continue, we fix $r\in(p_0,p)$ and split $K_N = \tilde K_N + [K_N-\tilde K_N]$.  Choosing $\sigma$ small, we may apply Proposition~\ref{P:Strichartz}
with exponent pair $(r,r)$ to obtain
$$
\bigl|\langle \chi_{\Omega_\omega}, \tilde K_N \chi_{\Omega_\omega}\rangle_{L^2_{t,x}}\bigr|\lesssim |\Omega_\omega|^{\frac2{r'}} N^{d-\frac{2(d+2)}{r}}.
$$
On the other hand, by \eqref{E:diff bound},
$$
\bigl|\langle \chi_{\Omega_\omega}, [K_N-\tilde K_N] \chi_{\Omega_\omega}\rangle_{L^2_{t,x}}\bigr|\lesssim |\Omega_\omega|^2 N^{d(1-\sigma)}.
$$
Combining these inequalities with \eqref{E:omE} we obtain
\begin{align*}
\lambda^2|\Omega_\omega|^2\lesssim |\Omega_\omega|^{\frac2{r'}} N^{d-\frac{2(d+2)}{r}}+ |\Omega_\omega|^2 N^{d(1-\sigma)}.
\end{align*}
We now choose $\delta\ll \frac{d\sigma}2$ so that the second term on the right-hand side of the inequality above is much smaller than the left-hand side.  Thus we deduce that
\begin{align*}
|\Omega|\leq 4|\Omega_\omega|\lesssim N^{\frac r2(d-\frac{2(d+2)}{r})}\lambda^{-r}.
\end{align*}
Recalling the definition of $\Omega$ and that $r\in(p_0,p)$, it then follows that
\begin{align*}
\int_{N^{\frac d2-\delta}}^{CN^{\frac d2}} p\lambda^{p-1}&\bigl|\{(t,x)\in [0,1]\times\T^d:\, \bigl|(e^{it\Delta}P_{\leq N} f)(x) \bigr|>\lambda\} \bigr|\, d\lambda\\
&\lesssim N^{\frac r2(d-\frac{2(d+2)}{r})} \int_{N^{\frac d2-\delta}}^{CN^{\frac d2}}\lambda^{p-1-r}\, d\lambda\\
&\lesssim N^{p(\frac d2-\frac{d+2}p)}.
\end{align*}
This bounds the contribution of large values of $\lambda$ to \eqref{to estimate} in an acceptable manner and so completes the proof of Theorem~\ref{T:Strichartz}.
\end{proof}

%%%%%%%%%%%%%%%%%%%%%%%%%%%%%%%%%%%%%%%%%%%%%%%%%%%%%%%%%%%%%%%%%%%%%%%%%%%%%%%%%%%%%%%%%%%%%%%%%%%%%%%%%%%%%%%%%%%%%%%%%%%%%%
\section{Bilinear Strichartz estimates}
%%%%%%%%%%%%%%%%%%%%%%%%%%%%%%%%%%%%%%%%%%%%%%%%%%%%%%%%%%%%%%%%%%%%%%%%%%%%%%%%%%%%%%%%%%%%%%%%%%%%%%%%%%%%%%%%%%%%%%%%%%%%%%

The purpose of this section is to prove a bilinear estimate (in all dimensions) that we will use in our treatment of the energy-critical problem in three dimensions.
The prior treatment \cite{HTT} of this problem used a trilinear estimate whose proof is much more complicated.  The idea of splitting into frequency cubes, which
we will also use, is just the first step in their proof.

\begin{lemma}[Bilinear Strichartz estimate]\label{L:bilinear1}
Fix $d\geq 3$ and $T\leq 1$.  Then for every $1\leq N_2\leq N_1$ we have
\begin{align}\label{L2 bilinear}
\|u_{N_1}v_{N_2}\|_{L_{t,x}^2([0,T)\times\T^d)}\lesssim N_2^{\frac{d-2}2}\|u_{N_1}\|_{Y^0([0,T))} \| v_{N_2}\|_{Y^0([0,T))}.
\end{align}
The implicit constant does not depend on $T$.
\end{lemma}

\begin{remark}
In the Euclidean setting one has the following stronger estimate:
\begin{align*}
\|u_{N_1}v_{N_2}\|_{L_{t,x}^2(\R\times\R^d)}\lesssim N_2^{\frac{d-1}2} N_1^{-\frac12} \|u_{N_1}\|_{Y^0} \| v_{N_2}\|_{Y^0}.
\end{align*}
No such estimate holds on the torus.  Indeed, choosing $u$ and $v$ to be linear solutions with characters as initial data, one can
see that no negative power of the higher frequency can appear on the RHS\eqref{L2 bilinear}.
\end{remark}

\begin{proof}
To prove \eqref{L2 bilinear}, we decompose $\R^d=\cup_j C_j$, where each $C_j$ is a cube of side-length $N_2$.  We write $P_{C_j}$ for the (sharp) Fourier
projection onto this cube.  As the spatial Fourier support of $(P_{C_j} u_{N_1})v_{N_2}$
is contained in a fixed dilate of the cube $C_j$, for each $j$, we deduce that
\begin{align*}
\|u_{N_1}v_{N_2}\|_{L_{t,x}^2([0,T)\times\T^d)}\lesssim \Bigl(\sum_j \bigl\|(P_{C_j}u_{N_1})v_{N_2}\bigr\|_{L_{t,x}^2([0,T)\times\T^d)}^2\Bigr)^{1/2}.
\end{align*}
Using the Strichartz inequality \eqref{UCStrichartz}, we estimate
\begin{align*}
\bigl\|(P_{C_j}u_{N_1})v_{N_2}\bigr\|_{L_{t,x}^2([0,T)\times\T^d)}
&\lesssim \|P_{C_j}u_{N_1}\|_{L_{t,x}^4([0,T)\times\T^d)}\|v_{N_2}\|_{L_{t,x}^4([0,T)\times\T^d)}\\
%&\lesssim N_2^{\frac{d-2}2}\|P_{C_j}u_{N_1}\|_{U^4_\Delta L^2}\|v_{N_2}\|_{U^4_\Delta L^2}\\
&\lesssim N_2^{\frac{d-2}2}\|P_{C_j}u_{N_1}\|_{Y^0}\|v_{N_2}\|_{Y^0}.
\end{align*}
Observing that
$$
\|u\|_{Y^0([0,T))}= \Bigl(\sum_j \bigl\|P_{C_j}u_{N_1}\bigr\|_{Y^0([0,T))}^2\Bigr)^{1/2},
$$
we immediately derive \eqref{L2 bilinear}.
\end{proof}

By further exploiting the ideas in \cite{HTT}, one can obtain a stronger bilinear Strichartz estimate.  We will not use this result in this paper and simply record the estimate for comparison.
In the case $d=4$, what follows is essentially \cite[Proposition~2.8]{HTT:4d}.  Their argument can be adapted to dimensions $d\geq 3$ because of the $L_{t,x}^4$ Strichartz estimate
given in Theorem~\ref{T:Strichartz}.

\begin{lemma}[Improved bilinear Strichartz estimate]\label{L:bilinear2}
Fix $d\geq 3$ and $T\leq 1$.  Then there exists $\delta>0$ such that for every $1\leq N_2\leq N_1$ we have
\begin{align*}
\|u_{N_1}v_{N_2}\|_{L_{t,x}^2([0,T)\times\T^d)}\lesssim N_2^{\frac{d-2}2}\bigl(\tfrac{N_2}{N_1}+\tfrac1{N_2} \bigr)^\delta\|u_{N_1}\|_{Y^0([0,T))} \| v_{N_2}\|_{Y^0([0,T))}.
\end{align*}
\end{lemma}

%%%%%%%%%%%%%%%%%%%%%%%%%%%%%%%%%%%%%%%%%%%%%%%%%%%%%%%%%%%%%%%%%%%%%%%%%%%%%%%%%%%%%%%%%%%%%%%%%%%%%%%%%%%%%%%%%%%%%%%%%%%%%%
\section{Well-posedness for the energy-critical NLS}\label{S:lowD}
%%%%%%%%%%%%%%%%%%%%%%%%%%%%%%%%%%%%%%%%%%%%%%%%%%%%%%%%%%%%%%%%%%%%%%%%%%%%%%%%%%%%%%%%%%%%%%%%%%%%%%%%%%%%%%%%%%%%%%%%%%%%%%
The main estimates needed to prove local well-posedness for the energy-critical NLS are contained in the following proposition.

\begin{proposition}\label{P:nonlinearity}
Fix $d\in\{3,4\}$.  Then for any $0<T\leq 1$,
\begin{align}\label{F estimate}
\Bigl\| \int_0^t e^{i(t-s)\Delta} F(u(s))\, ds\Bigr\|_{X^1([0,T])}\lesssim \|u\|_{X^1([0,T])}^{\frac{d+2}{d-2}}
\end{align}
and
\begin{align}\label{F diff estimate}
\Bigl\| \int_0^t e^{i(t-s)\Delta} \bigl[F(u+w)(s)&-F(u)(s)\bigr]\, ds\Bigr\|_{X^1([0,T])}\notag\\
&\lesssim \|w\|_{X^1([0,T])}\bigl(\|u\|_{X^1([0,T])}+\|w\|_{X^1([0,T])}\bigr)^{\frac{4}{d-2}}.
\end{align}
The implicit constants do not depend on $T$.
\end{proposition}

\begin{proof}
As \eqref{F estimate} follows from \eqref{F diff estimate} by taking $u\equiv0$, we will only treat the latter.
Throughout the proof of the proposition all spacetime norms will be taken on $[0,T]\times\T^d$.

Fix $N\geq 1$ and observe that $P_{\leq N} [F(u+w)-F(u)]\in L^1([0,T]; H^1(\T^d))$.  By duality (see Proposition 2.11 in \cite{HTT}),
\begin{align*}
\Bigl\| \int_0^t e^{i(t-s)\Delta} &P_{\leq N} \bigl[F(u+w)(s)-F(u)(s)\bigr]\, ds\Bigr\|_{X^1([0,T])}\\
&\leq \sup_{\|\tilde v\|_{Y^{-1}([0,T])}=1} \Bigl| \int_0^T \int_{\T^d} P_{\leq N} \bigl[F(u+w)(t)-F(u)(t)\bigr] \overline{\tilde v(t,x)}\, dx\, dt \Bigr|.
\end{align*}
Let $v:=\overline{P_{\leq N} \tilde v}$.  We will prove that
\begin{align}\label{nonlinear estimate}
\Bigl| \int_0^T \int_{\T^d}  & \bigl[F(u+w)(t)-F(u)(t)\bigr] v(t,x)\, dx\, dt \Bigr| \notag \\ 
&\lesssim \|v\|_{Y^{-1}([0,T])} \|w\|_{X^1([0,T])}\bigl(\|u\|_{X^1([0,T])}+\|w\|_{X^1([0,T])}\bigr)^{\frac{4}{d-2}}.
\end{align}
Estimate \eqref{F diff estimate} follows from this by letting $N\to \infty$.

A little combinatorics shows that \eqref{nonlinear estimate} follows from an estimate of the form
\begin{align}\label{5est}
\sum_{N_0\geq 1}\ \sum_{N_1\geq \cdots \geq N_{\frac{d+2}{d-2}}\geq 1}\biggl|\int_0^T\int_{\T^d} v_{N_0}(t,x) \prod_{j=1}^{\frac{d+2}{d-2}} & u^{(j)}_{N_j}(t,x)\,dx\, dt\biggr|\notag\\
&\lesssim \|v\|_{Y^{-1}}\prod_{j=1}^{\frac{d+2}{d-2}} \|u^{(j)}\|_{X^1([0,T])},
\end{align}
by choosing $u^{(j)}$ varying over the collection $\{u, \bar u, w, \bar w\}$.  The remainder of the proof is dedicated to the verification of \eqref{5est}.

\noindent\textbf{Case I: $d=3$}.  In order to have a non-zero contribution to LHS\eqref{5est}, the two highest frequencies must be comparable.  We distinguish two subcases.

\noindent\textbf{Case I.1: $N_0\sim N_1\geq \cdots\geq N_5$}.  Using H\"older, Lemma~\ref{L:bilinear1}, Bernstein, and Cauchy--Schwarz, we estimate
\begin{align*}
&\sum_{N_0\sim N_1\geq \cdots\geq N_5} \Bigl| \int_0^T \int_{\T^d}  v_{N_0}^{\vphantom{(}}(t,x) u_{N_1}^{(1)}(t,x)\ldots u_{N_5}^{(5)}(t,x)\, dx\, dt \Bigr|\\
&\lesssim \sum_{N_0\sim N_1\geq \cdots\geq N_5}\|v_{N_0}^{\vphantom{(}} u_{N_2}^{(2)}\|_{L_{t,x}^2} \|u_{N_1}^{(1)} u_{N_3}^{(3)}\|_{L_{t,x}^2}
	\|u_{N_4}^{(4)}\|_{L_{t,x}^\infty}\|u_{N_5}^{(5)}\|_{L_{t,x}^\infty}\\
&\lesssim \sum_{N_0\sim N_1\geq \cdots\geq N_5} N_2^{\frac12}\|v_{N_0}\|_{Y^0} \| u_{N_2}^{(2)}\|_{Y^0}N_3^{\frac12}\|u_{N_1}^{(1)}\|_{Y^0} \| u_{N_3}^{(3)}\|_{Y^0} N_4^{\frac12}\| u_{N_4}^{(4)}\|_{L_t^\infty H_x^1}N_5^{\frac12}\| u_{N_5}^{(5)}\|_{L_t^\infty H_x^1}\\
&\lesssim \sum_{N_0\sim N_1\geq \cdots\geq N_5}\tfrac{N_0N_4^{\frac12}N_5^{\frac12}}{N_1N_2^{\frac12}N_3^{\frac12}}\|v_{N_0}\|_{Y^{-1}} \|u_{N_1}^{(1)}\|_{Y^1}\| u_{N_2}^{(2)}\|_{Y^1}\| u_{N_3}^{(3)}\|_{Y^1} \| u_{N_4}^{(4)}\|_{Y^1}\| u_{N_5}^{(5)}\|_{Y^1}\\
&\lesssim \|u^{(4)}\|_{Y^1} \|u^{(5)}\|_{Y^1}\sum_{N_0\sim N_1\geq N_2\geq N_3} \bigl(\tfrac{N_3}{N_2}\bigr)^{\frac12}\|v_{N_0}\|_{Y^{-1}} \|u_{N_1}^{(1)}\|_{Y^1}\| u_{N_2}^{(2)}\|_{Y^1}\| u_{N_3}^{(3)}\|_{Y^1}\\
&\lesssim \|u^{(4)}\|_{Y^1} \|u^{(5)}\|_{Y^1} \Bigl(\sum_{N_0} \|v_{N_0}\|_{Y^{-1}}^{2}\Bigr)^{1/2}\Bigl(\sum_{N_1} \|u_{N_1}^{(1)}\|_{Y^1}^{2}\Bigr)^{1/2}\\
&\qquad\qquad\qquad \times \Bigl(\sum_{N_2\geq N_3} \bigl(\tfrac{N_3}{N_2}\bigr)^{\frac12}\|u_{N_2}^{(2)}\|_{Y^1}^{2}\Bigr)^{1/2}
		\Bigl(\sum_{N_3\leq N_2} \bigl(\tfrac{N_3}{N_2}\bigr)^{\frac12}\|u_{N_3}^{(3)}\|_{Y^1}^{2}\Bigr)^{1/2}\\
&\lesssim \|v\|_{Y^{-1}} \prod_{j=1}^5 \|u^{(j)}\|_{Y^1}.
\end{align*}
This settles Case~I.1 because $X^1\hookrightarrow Y^1$.

\noindent\textbf{Case I.2: $N_0\lesssim N_1\sim N_2\geq N_3\geq N_4\geq N_5$}.  In this subcase, we do not need a bilinear estimate, only the Strichartz inequalities proved
in Theorem~\ref{T:Strichartz}:
\begin{align*}
&\sum_{N_0\lesssim N_1\sim N_2\geq \cdots\geq N_5} \Bigl| \int_0^T \int_{\T^d}  v_{N_0}(t,x) u_{N_1}^{(1)}(t,x)\ldots u_{N_5}^{(5)}(t,x) \, dx\, dt \Bigr|\\
&\lesssim \sum_{N_0\lesssim N_1\sim N_2\geq \cdots\geq N_5} \|v_{N_0}\|_{L_{t,x}^4} \|u_{N_1}^{(1)} \|_{L_{t,x}^4} \| u_{N_2}^{(2)}\|_{L_{t,x}^4} \| u_{N_3}^{(3)}\|_{L_{t,x}^4} \|u_{N_4}^{(4)}\|_{L_{t,x}^\infty}\|u_{N_5}^{(5)}\|_{L_{t,x}^\infty}\\
&\lesssim \sum_{N_0\lesssim N_1\sim N_2\geq \cdots\geq N_5} \tfrac{N_0^{\frac54} N_4^{\frac12} N_5^{\frac12}}{N_1^{\frac34} N_2^{\frac34}N_3^{\frac34}}\|v_{N_0}\|_{Y^{-1}} \|u_{N_1}^{(1)}\|_{Y^1}\| u_{N_2}^{(2)}\|_{Y^1}\| u_{N_3}^{(3)}\|_{Y^1} \| u_{N_4}^{(4)}\|_{Y^1}\| u_{N_5}^{(5)}\|_{Y^1}\\
&\lesssim \|v\|_{Y^{-1}}\prod_{j=3}^5\|u^{(j)}\|_{Y^1} \sum_{N_1\sim N_2} \bigl(\tfrac{N_1}{N_2}\bigr)^{\frac12}\| u_{N_1}^{(1)}\|_{Y^1}\| u_{N_2}^{(2)}\|_{Y^1}\\
&\lesssim \|v\|_{Y^{-1}} \prod_{j=1}^5 \|u^{(j)}\|_{Y^1}.
\end{align*}
This completes the proof of the proposition in the case of three space dimensions.  

\noindent\textbf{Case II: $d=4$}.  Again  we distinguish two subcases: either $N_0\sim N_1\geq N_2\geq N_3$ or $N_0\lesssim N_1\sim N_2\geq N_3$.

\noindent\textbf{Case II.1: $N_0\sim N_1\geq N_2\geq N_3$}.  Simply applying Lemma~\ref{L:bilinear1} as in Case~I.1 does not succeed because one is unable to sum the lower two frequencies; instead, we will exploit the main idea of the proof.  As there, let $P_{C_j}$ denote the family of Fourier projections onto a tiling of cubes of size $N_2$.  We write $C_j\sim C_k$ if the sum set overlaps the Fourier support of $P_{\leq 2 N_2}$.  Observe that given $C_k$ there are a bounded number of $C_j\sim C_k$.    Thus
\begin{align*}
&\sum_{N_0\sim N_1\geq N_2\geq N_3} \Bigl| \int_0^T \int_{\T^d} v_{N_0}^{\vphantom{(}}(t,x) u_{N_1}^{(1)}(t,x)u_{N_2}^{(2)}(t,x) u_{N_3}^{(3)}(t,x) \, dx\, dt \Bigr|\\
&\lesssim \sum_{N_0\sim N_1\geq N_2\geq N_3} \ \sum_{C_j\sim C_k} \bigl\|P_{C_j} v_{N_0}^{\vphantom{(}} \bigr\|_{L^{10/3}_{t,x}}   \bigl\|P_{C_k} u_{N_1}^{(1)} \bigr\|_{L^{10/3}_{t,x}}
		\bigl\| u_{N_2}^{(2)} \bigr\|_{L^{10/3}_{t,x}} \bigl\| u_{N_3}^{(3)}\bigr\|_{L^{10}_{t,x}}\\
&\lesssim \sum_{N_0\sim N_1\geq N_2\geq N_3} \ \sum_{C_j\sim C_k} \tfrac{N_0}{N_1} \bigl(\tfrac{N_3}{N_2}\bigr)^{\!\frac25} \bigl\|P_{C_j} v_{N_0}^{\vphantom{(}} \bigr\|_{Y^{-1}}
		\bigl\|P_{C_k} u_{N_1}^{(1)} \bigr\|_{Y^1} \bigl\| u_{N_2}^{(2)} \bigr\|_{Y^1} \bigl\| u_{N_3}^{(3)}\bigr\|_{Y^1}\\
&\lesssim \|u^{(2)}\|_{Y^1}\|u^{(3)}\|_{Y^1}  \sum_{N_0\sim N_1} \tfrac{N_0}{N_1} \|v_{N_0}\|_{Y^{-1}} \|u_{N_1}^{(1)}\|_{Y^1}\\
&\lesssim \|v\|_{Y^{-1}} \prod_{j=1}^3 \|u^{(j)}\|_{Y^1}.
\end{align*}

\noindent\textbf{Case II.2: $N_0\lesssim N_1\sim N_2\geq N_3$}.  We argue in the same manner as Case~I.2:
\begin{align*}
\sum_{N_0\lesssim N_1\sim N_2\geq N_3} &\Bigl| \int_0^T \int_{\T^d}  u_{N_1}^{(1)}(t,x)u_{N_2}^{(2)}(t,x) u_{N_3}^{(3)}(t,x) v_{N_0}(t,x)\, dx\, dt \Bigr|\\
&\lesssim \sum_{N_0\lesssim N_1\sim N_2\geq N_3} \bigl\|v_{N_0}\bigr\|_{L^{10/3}_{t,x}} \bigl\|u_{N_1}^{(1)}\bigr\|_{L^{10/3}_{t,x}} \bigl\|u_{N_2}^{(2)}\bigr\|_{L^{10/3}_{t,x}}  \bigl\|u_{N_3}^{(3)}\bigr\|_{L^{10}_{t,x}} \\
&\lesssim \sum_{N_0\lesssim N_1\sim N_2\geq N_3} \tfrac{N_0^{\frac65}N_3^{\frac25}}{N_1^{\frac45}N_2^{\frac45}}
	\bigl\|v_{N_0}\bigr\|_{Y^{-1}} \bigl\| u_{N_1}^{(1)}\bigr\|_{Y^1} \bigl\|u_{N_2}^{(2)}\bigr\|_{Y^1}\bigl \| u_{N_3}^{(3)}\bigr\|_{Y^1}\\
&\lesssim \|v\|_{Y^{-1}}\|u^{(3)}\|_{Y^1}  \sum_{N_1\sim N_2} \bigl(\tfrac{N_1}{N_2}\bigr)^{\!\frac25}\|u_{N_1}^{(1)}\|_{Y^1} \|u_{N_2}^{(2)}\|_{Y^1}\\
&\lesssim \|v\|_{Y^{-1}} \prod_{j=1}^3 \|u^{(j)}\|_{Y^1}.
\end{align*}
This completes the proof of the proposition when $d=4$.
\end{proof}

\begin{proof}[Proof of Theorem~\ref{T:EC}]
We first consider the case of \emph{small initial data}.  Fix $d\in \{3,4\}$ and let $u_0\in H^1(\T^d)$ satisfy
\begin{align*}
\|u_0\|_{H^1(\T^d)}\leq \eta\leq \eta_0
\end{align*}
for a small $\eta_0=\eta_0(d)$ to be chosen later.

We first note that by conservation of mass and energy, it suffices to construct the solution to the initial-value problem \eqref{NLS} on the time interval $[0,1]$.
Indeed, by Sobolev embedding,
$$
\|f\|_{L^{\frac{2d}{d-2}}(\T^d)}\lesssim_d \|f\|_{H^1(\T^d)}
$$
and so, in both the defocusing and the focusing cases we have
$$
M(u) + E(u)=\int_{\T^d} \tfrac12|u_0(x)|^2+ \tfrac12|\nabla u_0(x)|^2 \pm \tfrac{d-2}{2d} |u_0(x)|^{\frac{2d}{d-2}}\, dx \sim \|u_0\|_{H^1(\T^d)}^2,
$$
provided $\eta_0(d)$ is chosen sufficiently small.  Using a continuity argument together with the conservation of mass and energy, we deduce that this equivalence holds at all times of existence, namely,
$$
M(u) + E(u)\sim \|u(t)\|_{H^1(\T^d)}^2.
$$
Thus, a simple iteration argument allows us to extend the local-in-time solution to a global-in-time solution.

To construct the solution to \eqref{NLS} on the time interval $[0,1]$, we use a contraction mapping argument.  More precisely, we will show that the mapping
\begin{align}\label{Phi(u)}
\Phi(u)(t):= e^{it\Delta}u_0 \mp i \int_0^t e^{i(t-s)\Delta}F(u(s))\, ds
\end{align}
is a contraction on the ball
$$
B:=\bigl\{ u\in X^1([0,1])\cap C_tH^1_x([0,1]\times\T^d):\, \|u\|_{X^1([0,1])}\leq 2\eta\bigr\}
$$
under the metric
$$
d(u,v):=\|u-v\|_{X^1([0,1])}.
$$

Using Proposition~\ref{P:nonlinearity}, we see that for $u\in B$,
\begin{align*}
\|\Phi(u)\|_{X^1([0,1])}&\leq \|e^{it\Delta}u_0\|_{X^1([0,1])} + \Bigl\| \int_0^t e^{i(t-s)\Delta}F(u(s))\, ds \Bigr\|_{X^1([0,1])}\\
&\leq \|u_0\|_{H^1(\T^d)} + C \|u\|_{X^1([0,1])}^{\frac{d+2}{d-2}}\leq \eta+C(2\eta)^{\frac{d+2}{d-2}}\leq 2\eta,
\end{align*}
provided $\eta_0$ is chosen sufficiently small.  This proves $\Phi$ maps the ball $B$ to itself.

To see that $\Phi$ is a contraction under the metric $d$, we apply Proposition~\ref{P:nonlinearity} to $u, v\in B$ to get
\begin{align*}
d(\Phi(u),\Phi(v))&\leq \Bigl\| \int_0^t e^{i(t-s)\Delta}\bigl[F(u(s))-F(v(s))\bigr]\, ds \Bigr\|_{X^1([0,1])}\\
&\lesssim \|u-v\|_{X^1([0,1])} \bigl( \|u\|_{X^1([0,1])}+\|v\|_{X^1([0,1])}\bigr)^{\frac4{d-2}}\\
&\lesssim d(u,v) (4\eta)^{\frac4{d-2}}\\
&\leq\tfrac12 d(u,v),
\end{align*}
provided $\eta_0$ is chosen sufficiently small.

This completes the discussion of small initial data.  We now turn to the statement in Theorem~\ref{T:EC} concerning \emph{large initial data}.

Let $u_0\in H^1(\T^d)$ with
$$
\|u_0\|_{H^1(\T^d)}\leq A
$$
for some $0<A<\infty$.  Let $\delta>0$ be a small number to be chosen later (depending on $A$) and let $N=N(u_0)\geq 1$ be such that
$$
\| P_{>N} u_0\|_{H^1(\T^d)}\leq \delta.
$$
We will show that the mapping $\Phi(u)$ defined in \eqref{Phi(u)} is a contraction on the ball
$$
B:=\bigl\{ u\in X^1([0,T])\cap C_tH^1_x([0,T]\times\T^d):\, \|u\|_{X^1([0,T])}\leq 2A, \, \|u_{>N}\|_{X^1([0,T])}\leq 2\delta\bigr\}
$$
under the metric
$$
d(u,v):=\|u-v\|_{X^1([0,T])},
$$
provided $T$ is chosen sufficiently small (depending on $A$, $\delta$, and $N$).
For the remainder of the proof, all space-time norms will be on $[0,T]\times\T^d$.

First we verify that $\Phi$ maps $B$ to itself.  Using Remark~\ref{R}, Proposition~\ref{P:nonlinearity}, and Bernstein, for $u\in B$ we estimate
\begin{align*}
\|\Phi(u)\|_{X^1}
&\leq \|e^{it\Delta}u_0\|_{X^1}+\Bigl\| \int_0^t e^{i(t-s)\Delta}F(u_{\leq N}(s))\, ds\Bigr\|_{X^1}\\
&\quad +\Bigl\| \int_0^t e^{i(t-s)\Delta}\bigl[F(u(s))-F(u_{\leq N}(s))\bigr]\, ds\Bigr\|_{X^1}\\
&\leq \|u_0\|_{H^1(\T^d)} + C\|F(u_{\leq N})\|_{L_t^1 H^1_x} + C\|u_{>N}\|_{X^1} \|u\|_{X^1}^{\frac4{d-2}}\\
&\leq A+ CT\|u_{\leq N}\|_{L_t^\infty H^1_x} \|u_{\leq N}\|_{L_{t,x}^\infty}^{\frac4{d-2}} + C(2\delta) (2A)^{\frac4{d-2}}\\
&\leq A+ CTN^2 (2A)^{\frac{d+2}{d-2}}+ C(2\delta) (2A)^{\frac4{d-2}}\\
&\leq 2A,
\end{align*}
provided $\delta$ is chosen small enough depending on $A$, and $T$ is chosen small enough depending on $A$ and $N$.

We decompose
$$
F(u)=F_1(u) + F_2(u) \qtq{where} F_1(u)=O\bigl(u_{>N}^2 u^{\frac{6-d}{d-2}}\bigr) \qtq{and} F_2(u)=O\bigl(u_{\leq N}^{\frac4{d-2}}u\bigr).
$$
Here, $O$ aggregates terms of similar structure, where factors may additionally have complex conjugates and/or further Littlewood--Paley projections.
Arguing similarly to the above, we estimate
\begin{align*}
&\|P_{>N}\Phi(u)\|_{X^1}\\
&\leq \|e^{it\Delta} P_{>N}u_0\|_{X^1}+\Bigl\| \int_0^t e^{i(t-s)\Delta}F_1(u(s))\, ds\Bigr\|_{X^1}+\Bigl\| \int_0^t e^{i(t-s)\Delta}F_2(u(s))\, ds\Bigr\|_{X^1}\\
&\leq \|P_{>N}u_0\|_{H^1(\T^d)} + C\|u_{>N}\|_{X^1}^2 \|u\|_{X^1}^{\frac{6-d}{d-2}} + C\|F_2(u)\|_{L_t^1H_x^1}\\
&\leq \delta + C(2\delta)^2 (2A)^{\frac{6-d}{d-2}}+ CT\Bigl[\|\nabla u\|_{L_t^\infty L^2_x} \|u_{\leq N}\|_{L_{t,x}^\infty}^{\frac4{d-2}} + \|u\|_{L_t^\infty L_x^{\frac{2d}{d-2}}}N\|u_{\leq N}\|_{L_t^\infty L_x^{\frac{4d}{d-2}}}^{\frac4{d-2}}\Bigr]\\
&\leq \delta + C(2\delta)^2 (2A)^{\frac{6-d}{d-2}} + CTN^2 (2A)^{\frac{d+2}{d-2}}\\
&\leq 2\delta,
\end{align*}
provided $\delta$ is chosen small enough depending on $A$, and $T$ is chosen small enough depending on $A$, $\delta$, and $N$.

Next, we prove that $\Phi$ is a contraction.  We again decompose $F=F_1+F_2$ and observe that
$$
F_1(u)-F_1(v)=O\Bigl((u-v)(u_{>N}+v_{>N})\bigl(u^{\frac{6-d}{d-2}}+v^{\frac{6-d}{d-2}}\bigr)\Bigr)
$$
and
$$
F_2(u)-F_2(v)=O\Bigl((u-v)\bigl(u_{\leq N}+v_{\leq N}\bigr)^{\frac4{d-2}}\Bigr)
	+ O\Bigl( (u_{\leq N}-v_{\leq N}) (u+v) \bigl(u_{\leq N}+v_{\leq N}\bigr)^{\frac{6-d}{d-2}}\Bigr).
$$
Employing Remark~\ref{R}, Proposition~\ref{P:nonlinearity}, and Bernstein as before, for $u,v\in B$ we estimate
\begin{align*}
&d\bigl(\Phi(u), \Phi(v)\bigr)\\
&\lesssim \|u-v\|_{X^1} \bigl(\|u_{>N}\|_{X^1} +\|v_{>N}\|_{X^1}\bigr)\bigl(\|u\|_{X^1} +\|v\|_{X^1}\bigr)^{\frac{6-d}{d-2}}+ \|F_2(u)-F_2(v)\|_{L_t^1H_x^1}\\
&\lesssim (4\delta) (4A)^{\frac{6-d}{d-2}} d(u,v) + T\|\nabla (u-v)\|_{L_t^\infty L^2_x} \bigl(\|u_{\leq N}\|_{L_{t,x}^\infty}+\|v_{\leq N}\|_{L_{t,x}^\infty}\bigr)^{\frac4{d-2}} \\
&\quad+ T\|u-v\|_{L_t^\infty L_x^{\frac{2d}{d-2}}}N\bigl(\|u_{\leq N}\|_{L_t^\infty L_x^{\frac{4d}{d-2}}}+\|v_{\leq N}\|_{L_t^\infty L_x^{\frac{4d}{d-2}}}\bigr)^{\frac4{d-2}}\\
&\quad+ T\bigl(\|\nabla u\|_{L_t^\infty L_x^2}+\|\nabla v\|_{L_t^\infty L_x^2}\bigr)\|u_{\leq N}-v_{\leq N}\|_{L_{t,x}^\infty}\bigl(\|u_{\leq N}\|_{L_{t,x}^\infty}+\|v_{\leq N}\|_{L_{t,x}^\infty}\bigr)^{\frac{6-d}{d-2}}\\
&\quad+ T\bigl(\|u\|_{L_t^\infty L_x^{\frac{2d}{d-2}}}+\|v\|_{L_t^\infty L_x^{\frac{2d}{d-2}}}\bigr)N\|u_{\leq N}-v_{\leq N}\|_{L_t^\infty L_x^{\frac{4d}{d-2}}}\\
&\qquad\qquad\qquad\qquad\qquad\qquad\times\bigl(\|u_{\leq N}\|_{L_t^\infty L_x^{\frac{4d}{d-2}}}+\|v_{\leq N}\|_{L_t^\infty L_x^{\frac{4d}{d-2}}}\bigr)^{\frac{6-d}{d-2}}\\
&\lesssim \bigl[(4\delta) (4A)^{\frac{6-d}{d-2}} + TN^2 (4A)^{\frac4{d-2}}\bigr]d(u,v) \\
&\leq \tfrac12d(u,v),
\end{align*}
provided $\delta$ is chosen small enough depending on $A$, and $T$ is chosen small enough depending on $A$ and $N$.

By the contraction mapping theorem, this allows us to construct a unique solution $u$ to \eqref{NLS} in the ball $B$.  To see that uniqueness holds in the larger class $X^1([0,T])\cap C_tH^1_x([0,T]\times\T^d)$, we need only observe that if $v\in X^1([0,T])\cap C_tH^1_x([0,T]\times\T^d)$ is a second solution to \eqref{NLS} with data $v(0)=u_0$, then there exists $N_0\geq 1$ such that
$$
\|v_{>N_0}\|_{X^1([0,T])}\leq 2\delta.
$$
Choosing the larger of $N$ and $N_0$, we find a new ball $B$ that contains both $u$ and $v$.  In this way, the contraction mapping argument guarantees $u=v$ on a possibly smaller interval $[0,T']$.  Iterating this argument yields uniqueness in the larger class.
\end{proof}

%%%%%%%%%%%%%%%%%%%%%%%%%%%%%%%%%%%%%%%%%%%%%%%%%%%%%%%%%%%%%%%%%%%%%%%%%%%%%%%%%%%%%%%%%%%%%%%%%%%%%%%%%%%%%%%%%%%%%%%%%%%%%%
\bibliographystyle{amsplain}
\bibliography{Strichartz}
%%%%%%%%%%%%%%%%%%%%%%%%%%%%%%%%%%%%%%%%%%%%%%%%%%%%%%%%%%%%%%%%%%%%%%%%%%%%%%%%%%%%%%%%%%%%%%%%%%%%%%%%%%%%%%%%%%%%%%%%%%%%%%

\end{document}